\documentclass[10pt]{article}
\usepackage{amssymb}
\usepackage{amsmath}
\usepackage{amsfonts}
\usepackage{amsthm}
\usepackage{color}
\usepackage[english]{babel}
\usepackage[affil-it]{authblk}

\newtheorem{theorem}{Theorem}
\newtheorem{lemma}{Lemma}
\newtheorem{proposition}{Proposition}

\title{On reconstruction of eigenfunctions of Johnson graphs}

\author{Konstantin Vorob'ev %
  \thanks{E-mail address: \texttt{vorobev@math.nsc.ru}}}
\affil{Sobolev Institute of Mathematics, Novosibirsk, Russia;\\ Novosibirsk State University, Novosibirsk, Russia}

\begin{document}

\maketitle

\begin{abstract}

In the present work we consider the problem of a reconstruction of eigenfunctions of the Johnson graph $J(n,w)$. We give necessary and sufficient numerical conditions for a unique reconstruction of an eigenfunction with given eigenvalue by its values on a sphere of given radius $r$ for $n$ big enough. We also provide examples of functions equal on the sphere but not equal on the full vertex set in the case of a failure of these conditions.  

\end{abstract}

\section{Introduction}

There is a classical problem of a reconstruction of a function with given properties defined on a set of vertices $V$ of some graph $G=(V,E)$ by its values on some $V'\subset V$. Typical questions in this area are as follows: is this reconstruction unique, what is minimum size of such a set? Typically, if this function corresponds to some regular or symmetric objects in a graph such as perfect codes, completely regular codes, coverings, then the problem may be reduced to investigating properties of eigenspaces of the graph $G$.
Thus the problem of a reconstruction of eigenvectors (eigenfunctions) from graph's eigenspaces by partial information plays important role in graph theory.

In Hamming graphs this problem was first considered in \cite{Avg}. In that work Avgustinovich showed that a $1$-perfect code is determined by its code vertices in the middle layer of $H(n,2)$. In \cite{AvgVas} this result was generalized for an eigenfunction with corresponding eigenvalue. Later the problem of partial and full reconstruction of an arbitrary eigenfunction of the Hamming graph $H(n,q)$ by its values on a sphere was considered by Vasil'eva in \cite{Vas2} and the author in \cite{Vorob} for $q=2$ and by Vasil'eva in \cite{Vasq} for $q> 2$.
The problem of a reconstruction of eigenfunctions and perfect colorings (equitable partitions) for transitive graphs was considered in \cite{AvgLis}.

In the present work we consider this problem for eigenfunctions of the Johnson graph $J(n,w)$. For given $r$, $w$, $n$ and an eigenvalue $\lambda$, we give necessary and sufficient numerical conditions of a unique reconstruction of an eigenfunction with given eigenvalue $\lambda$ by its values on a sphere or a ball of given radius $r$ for $n$ big enough.

\section{Preliminaries}

In our definitions and statements we follow the notation from \cite{VMV}. 

Let $G=(V,E)$ be an undirected graph. A function $f:V\rightarrow \mathbb{R}$ is called a $\lambda$-{\it eigenfunction} of $G$ if the following equality holds for any
 $x\in V$:

 $$\lambda f(x)=\sum_{y\in V:(x,y)\in E} f(y). $$
Equivalently, $f$ is a $\lambda$-eigenfunction of $G$ if its
vector of values $\overline{f}$ is an eigenvector of the adjacency
matrix $A$ of $G$ with eigenvalue $\lambda$ or $\overline{f}$ is
the all-zero vector, i.e. the following holds:
$$A\overline{f}=\lambda \overline{f}.$$
For $f:V\rightarrow \mathbb{R}$ and $V'\subseteq V$, we will denote by $f\big|_{V'}$ a function which is defined on the set $V'$ and coincides with $f$ on this set.     

The vertices of the Johnson graph $J(n,w)$, $n\ge 2w$ are the binary
vectors of length $n$ with $w$ ones, where two vectors are adjacent if they have exactly $w-1$ common ones. Let us denote by $wt(x)$ the number of ones in a vector $x$. Distance $d(x,y)$ between vertices $x$ and $y$ is defined as $\frac{1}{2}wt(x+y)$, where $x+y$ is a sum of vectors modulo $2$. It is easy to see that this distance is equivalent to the length of the shortest path in the graph. For a vertex $x_0$ of $J(n,w)$ a sphere with the center in $x_0$ is a set $S_r(x_0)=\{y\in J(n,w)|d(x_0,y)=r\}$ and a ball is a set $B_r(x_0)=\{y\in J(n,w)|d(x_0,y)\le r\}$.

The Johnson graph $J(n,w)$ is distance-regular (see, for example, \cite{BCN}) with $w+1$ distinct eigenvalues of its adjacency matrix $\lambda_i=(w-i)(n-w-i)-i$, $i=0,1, \dots w$ and corresponding multiplicities $\binom{n}{i}-\binom{n}{i-1}$.

It is known (see for example \cite{Delsarte}) that for any $\lambda_i(n,w)$-eigenfunction of $J(n,w)$ and vertex $x$
the sum of the values of $f$ on vertices at distance $k$ from $x$ can be expressed using the Eberlein polynomials and the value $f(x)$:
$$\sum_{y \in J(n,w), d(x,y)=k}f(y)=f(x)E_k(i,w,n),$$ 
where $E_k(i,w,n)= {\sum_{j=0}^{k}{(-1)^j{i \choose j}{w-i \choose k-j}{n-w-i \choose k-j}}}.$

The following example of so-called {\it radial} function will be useful for our future arguments.
\begin{lemma}\label{radial}
Let $x_0$ be a vertex of the graph $J(n,w)$, $i\in \{0,1,\dots,w\}$ and $f$ be a function defined on vertices of the graph as follows:

$$f(x)=\frac{E_{d(x,x_0)}(i,w,n)}{{w \choose i}{n-w \choose i}}.$$

Then $f$ is $\lambda_i(n,w)$-eigenfunction of $J(n,w)$.
\end{lemma}
The proof can be obtained directly from the definition of an eigenfunction or from the fact that the Johnson graph is vertex-transitive.

Let $f$ be a real-valued $\lambda_i(n,w)$-eigenfunction of
$J(n,w)$ for some $i\in \{0,1,\dots,w\}$ and $j_1,j_2\in
\{1,2,\dots,n\}$, $j_1<j_2$. A real-valued function
$f_{j_1,j_2}$ is defined as follows: for any vertex $y=(y_1,y_2, \dots
,y_{j_1-1},y_{j_1+1}, \dots ,y_{j_2-1},y_{j_2+1},\dots,y_n)$ of
$J(n-2,w-1)$
\begin{flushright}$f_{j_1,j_2}(y)=f(y_1,y_2, \dots
,y_{j_1-1},1,y_{j_1+1}, \dots ,y_{j_2-1},0,y_{j_2+1},\dots,y_n)\,$
\end{flushright} \begin{flushright}
 $-f(y_1,y_2, \dots ,y_{j_1-1},0,y_{j_1+1}, \dots ,y_{j_2-1},1,y_{j_2+1}, \dots ,y_n).$
\end{flushright}

\begin{lemma}\cite{VMV}\label{derivative}
 If f is a $\lambda_i(n,w)$-eigenfunction of
$J(n,w)$ then $f_{j_1,j_2}$ is a
$\lambda_{i-1}(n-2,w-1)$-eigenfunction of $J(n-2,w-1)$.
\end{lemma}

Actually there is also a way to construct an eigenfunction in a higher Johnson graph.
\begin{lemma}\label{integral}
 Let f be a $\lambda_i(n,w)$-eigenfunction of
$J(n,w)$ and $g$ be a function on vertices $J(n+2,w+1)$ defined as follows: 

   $$g(y_1,y_2,\dots, y_{n+1}, y_{n+2})= \begin{cases}
     f(y_1,y_2,\dots,y_n), y_{n+1}=1 \text{ and } y_{n+2}=0 \\
   -f(y_1,y_2,\dots,y_n), y_{n+1}=0 \text{ and } y_{n+2}=1 \\
     0, \text{otherwise.}
    \end{cases}$$
Then $g$ is a $\lambda_{i+1}(n+2,w+1)$-eigenfunction of
$J(n+2,w+1)$.
\end{lemma}
The proof can be obtained by direct using the definition of an eigenfunction. 

Let us consider a function $f_0^{i,w,n}:J(n,w)\rightarrow \mathbb{R}$ defined in the following way: 

\begin{flushleft}
$f_0^{i,w,n}(x_1,x_2,\dots, x_i,x_{i+1},x_{i+2},x_{2i},x_{2i+1},\dots, x_n)=$\end{flushleft} $$\begin{cases}
    (-1)^{x_1+x_2+\dots+x_i}, x_{j}+x_{j+i}=1 \text{ for all } j=1,2,\dots i  \\
     0, \text{otherwise.}
    \end{cases}
$$

This function is a $\lambda_i(n,w)$-eigenfunction of
$J(n,w)$ \cite[Proposition 1]{VMV}(our function is a special case of $f^{i,w,n}$ for $M=\{1,2,\dots,i\}$, $M'=\{i+1,i+2,\dots,2i\}$)

\begin{proposition}\label{examplesof0} Let $n,w,i,r$ be non-negative integers such that $n\ge 2w$, $w\ge i$ and $w\ge r$.
\begin{enumerate}
\item Let $n\ge w+2r+2$, $i>r$ and $x=(x_1,x_2,\dots, x_n)$ be a vertex of $J(n,w)$ such that $x_1=x_2=\dots=x_{r+1}=0$ and $x_{i+1}=x_{i+2}=\dots=x_{i+r+1}=0$. Then $f_0^{i,w,n}\big|_{B_r(x)}\equiv 0$.
\item Let $i\le r$, $r>w-i$ and $x=(x_1,x_2,\dots, x_n)$ be a vertex of $J(n,w)$ such that $x_1=x_2=\dots=x_{w-r+1}=1$ and $x_{i+1}=x_{i+2}=\dots=x_{i+w-r+1}=1$. Then $f_0^{i,w,n}\big|_{S_r(x)}\equiv 0$.

\end{enumerate}
\end{proposition}
\begin{proof}
 In the first case given inequalities guarantee us that $n-2(r+1)\ge w$ and consequently the existence of a vertex $x$. By definition of $f_0^{i,w,n}$ this function has non-zero values only on vectors with exactly one $1$ in $i$ pairs of coordinate positions $(x_j,x_{i+j})$, $j=1,2,\dots, i$. Therefore any vertex at distance $r<i$ will have at least one such pair with two zeros.   
 
 In the second case given inequalities guarantee us that $w-r+1\le i$ and $w\ge 2(w-r+1)$ and consequently the existence of vertex $x$. First $w-r+1$ pairs $(x_j,x_{i+j})$, $j=1,2,\dots, w-r+1$ of coordinate positions of $x$ are equal to $(1,1)$. Therefore any vertex at distance $r>w-i$ will have at least one such pair with two zeros.   
\end{proof}

Let $f$ be a real-valued function defined on the vertices of the
Johnson graph $J(n,j)$. The function $I^{j,w}(f)$ on the
vertices of $J(n,w)$ is defined as follows:
$$I^{j,w}(f)(x)=\sum_{y, wt(y)=j, d(x,y)=|w-j|} f(y). $$
The function $I^{i,w}(f)$ is called {\it induced} in $J(n,w)$ by
$f$ \cite{AvgMog}. In this work we will need next properties of the operator $I^{j,w}$.

 \begin{theorem}\label{i_f}\cite{VMV}

1. Let $f$ be a $\lambda$-eigenfunction of $J(n,j)$. Then if
$j\leq w$ then $I^{j,w}(f)$ is a
$(\lambda+(w-i)(n-i-w))$-eigenfunction of $J(n,w)$.

2. Let $f$ be a real-valued function on the vertices of $J(n,w)$.
Then $I^{w,w-1}(f)\equiv 0$ iff $f$ is a $(-w)$-eigenfunction.
  \end{theorem}

Let us notice that the item $1$ from the Theorem \ref{i_f} means that the inducing operator maps an eigenfunction corresponding to an eigenvalue $\lambda_i(n,j)$ to an eigenfunction corresponding to an eigenvalue $\lambda_i(n,w)$. In other words, this operator saves the number of an eigenvalue. It is known, that it is also true for $j\geq w$.

It is easy to see that for $j\le w$, we have $$(w-j)!I^{j,w}(f)=I^{w-1,w}(\ldots (I^{j+1,j+2}(I^{j,j+1}(f)))),$$  $$(w-j)!I^{w,j}(f)=I^{j+1,j}(\ldots (I^{w-1,w-2}(I^{w,w-1}(f)))).$$ 
Based on this fact, properties of the eigenspaces of the Johnson graph and the inducing operator, it is possible to prove the following statement.

\begin{proposition}\label{inducing}
Let $n,w,w',i$ be nonnegative integers, $n\ge 2w$, $n\ge 2w'$, $i\le w$, $i\le w'$. Let $f$ be a $\lambda_i(n,w)$-eigenfunction of $J(n,w)$.
Then $\alpha f=I^{w',w}(I^{w,w'}(f))$ for some nonzero $\alpha$.
\end{proposition}

Since an eigenspace is a linear space, the existence of two different $\lambda$-eigenfunctions which are equal on some set is equivalent to the existence of a nonzero $\lambda$-eigenfunction which is the all-zero function on this set. In the following section we will focus on spheres and balls in Johnson graph and the problem of existence of such functions for these sets.

\section{Main result}

First of all we will prove that if there are two $\lambda_i(n,w)$-eigenfunctions $f_1,f_2$ of the graph $J(n,w)$ which coincide on a ball of radius $i$ then $f_1 \equiv f_2$ on the whole vertex set.

\begin{theorem}\label{mainball}
	
	Let $i,w,r,n$ be integers, $w> r \ge i\ge 0$. Let $f$ be a $\lambda_i(n,w)$-eigenfunction of $J(n,w)$ such that $f\big|_{B_r(x_0)} \equiv 0$ for some vertex $x_0\in J(n,w)$. Then $f\equiv 0$.

\end{theorem}

\begin{proof}

 Suppose that $f\not\equiv 0$. Let $r'$ be the smallest integer such that there exist $x\in J(n,w)$, $d(x_0,x)=r'$ and $f(x)\neq 0$. By the condition of the theorem it is known that $r'> i$. 
 
 Without loss of generality we can take $$x_0=(\underbrace{1,\dots,1}_{w},\underbrace{0,\dots,0}_{n-w}),
 x=(\underbrace{1,\dots,1}_{w-r'},\underbrace{0,\dots,0}_{r'},\underbrace{1,\dots,1}_{r'},\underbrace{0,\dots,0}_{n-w-r'}).$$
 
  Consider the function $f_{w+1,w-r'+1}$. By Lemma \ref{derivative} this function is a $\lambda_{i-1}(n-2,w-1)$-eigenfunction of $J(n-2,w-1)$. After deleting two coordinates we obtain the set $\{1,2,3,\dots,w-r',w-r'+2,\dots,w,w+2,\dots, n\}$. By repeating this procedure for pairs of coordinates $(w+j, w-r'+j)$ for $j=2,3,\dots, i+1$ we obtain the function $q:J(n-2i-2,w-i-1)\to \mathbb{R}$ such that: $$q=(\dots((f_{w+1,w-r'+1})_{w+2,w-r'+2})\dots)_{w+i+1,w-r'+i+1}.$$ By Lemma \ref{derivative} this function is a $\lambda_{-1}(n-2i-2,w-i-1)$-eigenfunction of $J(n-2i-2,w-i-1)$, in other words, just the all-zero function. On the other hand,  $q(\underbrace{1,\dots,1}_{w-r'},\underbrace{0,\dots,0}_{r'-i-1},\underbrace{1,\dots,1}_{r'-i-1},\underbrace{0,\dots,0}_{n-w-r'})$ equals a linear combination of values of $f$. 
  
  All vectors except $x$ in this combination are elements of the $B_{r'-1}(x_0)$, so we conclude that  $q(\underbrace{1,\dots,1}_{w-r'},\underbrace{0,\dots,0}_{r'-i-1},\underbrace{1,\dots,1}_{r'-i-1},\underbrace{0,\dots,0}_{n-w-r'})=f(x)\neq 0$, which contradicts the fact that $q$ is the all-zero function.

\end{proof}

As we see for the case of a ball $B_r(x_0)$, $r>i$, if the reconstruction of the function by its values on this set is possible then it is always unique. For the case of a sphere $S_r(x_0)$, the problem becomes more complicated and an answer depends on calculating some polynomials.

Let us define the following polynomials:

$$F_1(k_1,k_2,i,r,w,n)=\sum_{s=0}^{k_2-k_1}{\frac{{r-k_1 \choose s}{n-w-k_2-r \choose {k_2-k_1-s}}Eb_{r-k_1-s}(i-(k_1+k_2),w-2k_1,n-2(k_1+k_2))}{{w-2k_1 \choose {r-k_1-s}}{n-w-2k_2 \choose {r-k_1-s}}}},$$
for $k_1\le k_2$ and

    $$F_2(k_1,k_2,i,r,w,n)= 
    \sum_{s=0}^{k_1-k_2}{\frac{{w-r-k_1 \choose s}{r-k_2 \choose {k_1-k_2-s}}Eb_{r-k_1+s}(i-(k_1+k_2),w-2k_1,n-2(k_1+k_2))}{{w-2k_1 \choose {r-k_1+s}}{n-w-2k_2 \choose {r-k_1+s}}}},$$
for $k_1 \ge k_2$.

\begin{theorem}\label{mainsphere}
	
	Let $i,w,r,n$ be integers, $\frac{n}{2}\ge w\ge i\ge 0$ and $n\ge \max(w+r+i,2w,w+2r+2)$, $x_0$ be a vertex of $J(n,w)$. Then there does not exist a $\lambda_i(n,w)$-eigenfunction $f$ of $J(n,w)$ such that $f\big|_{S_r(x_0)} \equiv 0$ and $f \not \equiv 0$ if and only if the following inequalities hold:
	
	\begin{enumerate}
		\item 	$i \le r \le w-i$
		\item 	$ F_1(k_1,k_2,i,r,w,n)\neq 0$, $k_1+k_2=0,1,\dots, i-1$, $k_1=0,1,\dots, \left \lfloor{\frac{k_1+k_2}{2}}\right \rfloor$
		\item 	$ F_2(k_1,k_2,i,r,w,n)\neq 0$, $k_1+k_2=0,1,\dots, i-1$, $k_1=\left \lceil{\frac{k_1+k_2}{2}}\right \rceil ,\left\lceil{\frac{k_1+k_2}{2}}\right \rceil+1,\dots, k_1+k_2$.
	\end{enumerate}
	
\begin{proof}
	Sufficiency. 
	Suppose that inequalities 1-3 hold but there exist a $\lambda_i(n,w)$-eigenfunction $f$ of $J(n,w)$ such that $f\big|_{S_r(x_0)} \equiv 0$ and $f \not \equiv 0$. 
	
	Without loss of generality we have $x_0=(\underbrace{1,\dots,1}_{w},\underbrace{0,\dots,0}_{n-w})$.
	
	Let us consider functions $f_{j_1,j_2}$ for $j_1,j_2\in \{1,2,\dots,w\}$, $j_1\neq j_2$. If for some such a pair $l_1$, $l_2$ the function $f_{l_1,l_2}$ is not the all-zero function then we take $f^{(1)}\equiv f_{l_1,l_2}$ and repeat the procedure for $j_1,j_2\in \{1,2,\dots,w\}\setminus \{l_1,l_2\}$. Suppose that this procedure has worked $k_1$ times and we can not continue. After an appropriate permutation of coordinate positions from the set $1,2,\dots,w$ we obtain the function $f^{(k_1)}\equiv (\dots(f_{1,k_1+1})_{2,k_1+2}\dots)_{k_1,2k_1}$ and for every pair $j_1$, $j_2$ of different integers from $\{2k_1+1,2k_1+2,\dots, w\}$ we have that $f^{(k_1)}_{j_1,j_2}\equiv 0$.    
	
	Now we repeat the whole process described above for the function $f^{(k_1)}$ and a set of coordinate positions $w+1,w+2,\dots, n$. Suppose that the procedure has worked $k_2$ times and we can not continue. After an appropriate permutation of coordinate positions from the set $w+1,w+2,\dots,n$ we obtain the function $f^{(k_1+k_2)}\equiv (\dots((((\dots(f_{1,k_1+1})_{2,k_1+2}\dots)_{k_1,2k_1})_{w+1,w+k_2+1})_{w+2,w+k_2+2})\dots)_{w+k_2,w+2k_2}$ and for every pair of different integers $j_1$, $j_2$ from $\{w+2k_2+1,w+2k_2+2,\dots, n\}$ we have that $f^{(k_1+k_2)}_{j_1,j_2}\equiv 0$. By construction of $f^{(k_1+k_2)}$ one can see that for every pair of different integers $j_1$, $j_2$ from $\{2k_1+1,2k_1+2,\dots, w\}$ the equality $f^{(k_1+k_2)}_{j_1,j_2}\equiv 0$ holds too. 
	
	In other words, for $x\in J(n-2k_1-2k_2,w-k_1-k_2)$ the value $f^{(k_1+k_2)}(x)$ depends only on the number of ones of $x$ in the set $\{2k_1+1,2k_1+2,\dots, w\}$ and does not depend on a distribution of ones of $x$ inside the sets $\{2k_1+1,2k_1+2,\dots, w\}$ and $\{w+2k_2+1,w+2k_2+2,\dots, n\}$.
	

	By Lemma \ref{derivative} we know that $f^{(k_1+k_2)}$ is a $\lambda_{i-k_1-k_2}(n-2k_1-2k_2,w-k_1-k_2)$-eigenfunction of $J(n-2k_1-2k_2,w-k_1-k_2)$. Therefore $k_1+k_2\le i$.

    Let us consider a vertex of $J(n,w)$ $$y=(\underbrace{1,\dots,1}_{k_1},\underbrace{0,\dots,0}_{k_1},\underbrace{1,\dots,1}_{w-r-k_1},\underbrace{0,\dots,0}_{r-k_1},\underbrace{1,\dots,1}_{k_2},\underbrace{0,\dots,0}_{k_2},\underbrace{1,\dots,1}_{r-k_2},\underbrace{0,\dots,0}_{n-w-r-k_2})$$
	and vertices of $J(n-2(k_1+k_2),w-(k_1+k_2))$ $$z=(\underbrace{1,\dots,1}_{w-r-k_1},\underbrace{0,\dots,0}_{r-k_1},\underbrace{1,\dots,1}_{r-k_2},\underbrace{0,\dots,0}_{n-w-r-k_2}), 
	x_0=(\underbrace{1,\dots,1}_{w-2k_1},\underbrace{0,\dots,0}_{n-w-2k_2}).$$
	The fact that $k_1\le i$ and $k_1\le i$ together with inequalities from the statement of the theorem guarantees us that all lengths of intervals in $y$ and $z$ are nonnegative. 
	
	Let us notice that by construction $f^{(k_1+k_2)}(z)$ is a linear combination of values of $f$ in vertices from the set $S_r(x_0)$, i.e. $f^{(k_1+k_2)}(z)=0$.
	
	In case of equality $k_1+k_2=i$, a function $f^{(k_1+k_2)}$ is a $\lambda_{0}(n-2i,w-i)$-eigenfunction, i.e. a nonzero constant function. Hence $f^{(k_1+k_2)}(z)\neq 0$ and we get a contradiction. 
	
	So we conclude that $k_1+k_2 < i$. Let us consider a function $h:J(n-2k_1-2k_2,w-2k_1)\rightarrow \mathbb{R}$ defined as follows $h=I^{w-k_1-k_2,w-2k_1}(f^{(k_1+k_2)})$. As it was mentioned above the value $f^{(k_1+k_2)}(x)$ depends only on the number of ones of $x$ in the set $\{2k_1+1,2k_1+2,\dots, w\}$. This fact implies that $h$ is a radial function (see Lemma \ref{radial}) with a center $x_0'$ having ones in coordinate positions $\{2k_1+1,2k_1+2,\dots, w\}$. 
	
	If $h(x_0')=0$ then as a radial function $h$ must be the all-zero function. Therefore by Lemma \ref{inducing} we have $h\not\equiv 0$, so  for $x\in J(n-2(k_1+k_2),w-2k_1)$, $d(x_0,x)=j$, we have $$h(x)=h(x_0')\frac{Eb_j(i-(k_1+k_2),w-2k_1,n-2(k_1+k_2))}{\binom{w-2k_1}{j}\binom{n-w-2k_2}{j}}.$$ 
    By Lemma \ref{inducing} we have $f^{(k_1+k_2)}=\alpha I^{w-2k_1,w-k_1-k_2}(I^{w-k_1-k_2,w-2k_1}(f^{(k_1+k_2)}))$ for some nonzero $\alpha$. In other words, $f^{(k_1+k_2)}=\alpha I^{w-2k_1,w-k_1-k_2}(h)$. 
    As we noted before $f^{(k_1+k_2)}(z)=0$. Let us find this value in terms of values of $h$. There are two different cases $k_1\le k_2$ and $k_1\ge k_2$  By direct calculation we have
    that

    \begin{flushleft}
      $f^{(k_1+k_2)}(z)=$
    \end{flushleft}
    
       \begin{flushright}
          $ \alpha h(x_0')\begin{cases}
              \sum_{s=0}^{k_2-k_1}{\frac{{r-k_1 \choose s}{n-w-k_2-r \choose {k_2-k_1-s}}Eb_{r-k_1-s}(i-(k_1+k_2),w-2k_1,n-2(k_1+k_2))}{{w-2k_1 \choose {r-k_1-s}}{n-w-2k_2 \choose {r-k_1-s}}}}, k_1\le k_2\\
              \sum_{s=0}^{k_1-k_2}{\frac{{w-r-k_1 \choose s}{r-k_2 \choose {k_1-k_2-s}}Eb_{r-k_1+s}(i-(k_1+k_2),w-2k_1,n-2(k_1+k_2))}{{w-2k_1 \choose {r-k_1+s}}{n-w-2k_2 \choose {r-k_1+s}}}},\text{otherwise.}
              \end{cases}$
        \end{flushright}
   Therefore at least one value among $F_1(k_1,k_2,i,r,w,n)$ and $F_2(k_1,k_2,i,r,w,n)$ for some $k_1$, $k_2$ is a zero and we get a contradiction.

	Necessity. 
	If the inequality $i \le r \le w-i$ is false then by Proposition \ref{examplesof0} there exist a non-zero $\lambda_i(n,w)$-eigenfunction $f$ which is the all-zero function on a sphere of radius $r$. 
	
	In case of $F_1(k_1,k_2,i,r,w,n)=0$ or $F_2(k_1,k_2,i,r,w,n)=0$ for some $k_1$, $k_2$, an example can be constructed as follows. Let us take a function $f^{(k_1+k_2)}$ which was built in the first part of the proof as a result of inducing operator applied to the radial function $h$. Then by using $k_1$ times Lemma \ref{integral} we add coordinate positions $\{j,k_1+j\}$ for $j=1,2,\dots, k_1$ and by $k_2$ times we add coordinate positions  $\{w+j,w+k_2+j\}$ for $j=1,2,\dots, k_2$. As a result we build a $\lambda_i(n,w)$-eigenfunction of $J(n,w)$ which is the all-zero function on a sphere of radius $r$.  
	
\end{proof}

\end{theorem}

\section{Discussion}

In this paper we find a criterion of unique reconstruction of an $\lambda_i(n,w)-$eigenfunction of a Johnson $J(n,w)$ graph by its values on a sphere of a given radius $r$ for $n\ge \max(w+r+i,2w,w+2r+2)$. Consequently, for some values of $r$, $w$, $i$ and small $n$, the answer remains unknown. This criterion is based on zeros of some large polynomial containing Eberlein polynomials. 

As it was shown in the Theorem \ref{mainball}, every $\lambda_i(n,w)$-eigenfunction is determined by its values in a ball of radius $i$. Similar statement can be proved for Hamming graphs $H(n,q)$ (for these graphs by $i$-th eigenvalue one means $(q-1)n-qi$). It would be interesting to check this statement for other families of distance regular graphs, for example for Grassman, Doob and bilinear forms graphs.

\section{Acknowledgements}
 This paper is an extended version of the talk given by the author at the Second Russian-Hungarian Combinatorial Workshop held in Budapest, June 27-29, 2018. The reported study was funded by RFBR according to the research project 18-31-00126.

\end{document}